\documentclass[12pt,a4paper]{article}
\usepackage{amsmath, amsfonts, ifthen, latexsym, amssymb, amsthm, amscd}
\usepackage{amsrefs,mdframed, etoolbox}

\usepackage[utf8]{inputenc}

\usepackage[margin=1in]{geometry}

\usepackage[shortlabels]{enumitem}
\usepackage{graphicx,xcolor}
\usepackage{url}
\usepackage{hyperref}
\hypersetup{colorlinks=true,citecolor=blue,filecolor=blue,linkcolor=blue,urlcolor=blue}
\usepackage[title]{appendix}

\newtheorem{theorem}{Theorem}[section]
\newtheorem*{theorem*}{Theorem}
\newtheorem{lemma}[theorem]{Lemma}
\newtheorem*{lemma*}{Lemma}
\newtheorem{claim}[theorem]{Claim}
\newtheorem{proposition}[theorem]{Proposition}

\newtheorem{conjecture}[theorem]{Conjecture}

\theoremstyle{definition}
\newtheorem{definition}[theorem]{Definition}
\newtheorem{question}[theorem]{Question}

\newtheorem{remark}[theorem]{Remark}

\newcommand{\al}{{\alpha}}

\newcommand{\eps}{{\varepsilon}}
\newcommand{\de}{{\delta}}

\newcommand{\ga}{{\gamma}}
\newcommand{\Ga}{{\Gamma}}

\renewcommand{\phi}{{\varphi}}

\renewcommand{\ge}{\geqslant}
\renewcommand{\le}{\leqslant}

\newcommand\RR{\mathbb R}

\newcommand\NN{\mathbb N}

\newcommand{\KK}{\mathbb K}
\newcommand{\MM}{\mathbb M}

\newcommand{\UU}{\mathbb U}

\renewcommand{\cal}[1]{{\mathcal #1}}

\newtoggle{no_cases}
\toggletrue{no_cases}
\newcommand{\case}[2][]{%
    \iftoggle{no_cases}{%
        \left\{\begin{array}{ll}#2 & #1
    }{%
        \\#2 & #1
    }%
    \togglefalse{no_cases}
}
\newcommand{\esac}{%
    \end{array}\right.
    \toggletrue{no_cases}
}

\provideboolean{theno_parts}
\setboolean{theno_parts}{true}
\renewcommand{\part}[1][(a)]{%
    \ifthenelse{\boolean{theno_parts}}{%
        \begin{enumerate}[#1]%
            \item
    }{%
        \item
    }
    \setboolean{theno_parts}{false}%
}
\newcommand{\trap}{%
    \ifthenelse{\boolean{theno_parts}}{}{%
        \end{enumerate}
    }
    \setboolean{theno_parts}{true}
}

\newcommand{\newop}[2]{%
    \expandafter\def\csname #1\endcsname{\operatorname{#2}}
}

\newop{Iso}{Iso}
\newop{inter}{int}
\newop{Id}{Id}

\title{On random compact sets, equidecomposition, and domains of expansion in $\RR^3$}
\author{Tomasz Cieśla\thanks{T.C.~was supported by the ERC Starting Grant ``Limits of Structures in Algebra and Combinatorics'' No. 805495}, Łukasz Grabowski\thanks{Ł.G.~was partially supported by the ERC Starting Grant ``Limits of Structures in Algebra and Combinatorics'' No. 805495}}
\date{}

\begin{document}
\maketitle
\abstract{
We study random compact subsets of $\RR^3$ which can be described as  "random Menger sponges".  We use those random sets to construct a pair of compact sets $A$ and $B$ in $\RR^3$ which are of the same positive measure, such that $A$ can be covered by finitely many translates of $B$, $B$ can be covered by finitely many translates of $A$, and yet $A$ and $B$ are not equidecomposable.  Furthermore, we construct the first example of a compact subset of $\RR^3$ of positive measure which is not a domain of expansion. This answers a question of Adrian Ioana.}

\setlength{\parskip}{0.2em}

\newop{logrange}{logrange}
\newop{range}{range}
\tableofcontents
\section{Introduction}

We recall that two sets $A,B\subset \RR^n$ are \emph{equidecomposable} if we can write $A= \bigsqcup_{i=1}^k A_i$, $B= \bigsqcup_{i=1}^k B_i$ and $B_i = \ga_i(A_i)$ for some sets $A_i,B_i\subset \RR^n$ and some isometries $\ga_i\in \Iso(\RR^n)$. 

Since the work of Banach and Tarski~\cite{BanachTarski} on equidecompositions of sets in $\RR^3$, there has been a very considerable amount of interest in various equidecomposition problems both in Euclidean spaces and in other spaces equipped with group actions. Relatively recent developments include Baire improvements of the Banach-Tarski theorem in~\cite{MR1227475},  measurable and Borel improvements (resp.~\cite{MR3612006} and \cite{MR3702673}) of the seminal work of Laczkovich~\cite{MR1037431} about equidecompositions of sets in $\RR^n$, and a substantial progress on the Gardner conjecture (\cite{ciesla2019measurable} and~\cite{kun2021gardners}). For a more complete discussion we refer the reader to the monograph~\cite{MR3616119}.

In this article we are motivated by the following general question.

\begin{question}\label{intro-q1} Given a compact set $A\subset \RR^3$ of positive measure, can we describe all other compact sets $B\subset \RR^3$ which are equidecomposable with $A$?
\end{question}

The analogous problem in $\RR^2$ seems to be very difficult, even when we restrict attention to some natural subfamilies of compact sets. Informally speaking, the fact that the group $\Iso(\RR^2)$ is amenable makes it ``difficult'' for sets to be equidecomposable. For example, Laczkovich~\cite{MR1992530} constructed  a continuum of Jordan domains in $\RR^2$, each of measure $1$, which are pairwise non-equidecomposable, answering a question posed by Mycielski~\cite{MR1538275}. 

In $\RR^3$, it is substantially  ``easier'' for sets to be equidecomposable. For example, Banach and Tarski~\cite{BanachTarski} showed that any two bounded subsets of $\RR^3$ with non-empty interiors are equidecomposable. Let us describe a conjectural answer to a variant of Question~\ref{intro-q1} which seemed plausible to us and which motivated this work.

Let us start with  some key definitions. For $A,B\subset \RR^3$, we say that $A$ \emph{covers} $B$ if $B\subset 
\bigcup_{i=1}^k \ga_i(A)$, for some $\ga_1,\ldots, \ga_k\in 
\Iso(\RR^3)$. We say that $A$ and $B$ \emph{cover each other} if $A$ 
covers $B$ and $B$ covers $A$. We note that covering each other is a 
necessary condition for the existence of an equidecomposition.

If $A\subset \RR^3$ is (Lebesgue) measurable then we say that $A$ is \emph{non-negligible} if $\mu(A)>0$. If $A$ and $B$ are measurable subsets of $\RR^3$ then we say that they are \emph{essentially equidecomposable} (resp. \emph{essentially cover each other}) if there exists measurable sets $A', B'$ such that $\mu(A\triangle A') = \mu(B\triangle B') = 0$, and furthermore $A'$ and $B'$ are equidecomposable (resp. $A'$ and $B'$ cover each other). We note that essentially covering each other is a necessary condition for the existence of an essential equidecomposition.

If $A,B\subset \RR^3$ are measurable sets then we say that they are essentially Lebesgue equidecomposable if they are essentially equidecomposable and the parts in the essential equidecomposition can be chosen to be measurable. We note that the condition $\mu(A)=\mu(B)$ is necessary for the existence of an essential Lebesgue decomposition.

Finally, if $A$ is a bounded non-negligible measurable subset of $\RR^3$ then we say that $A$ is a \emph{domain of expansion} if there exists a finite set $S\subset \Iso(\RR^3)$ and $\eps>0$ such that for any measurable $U\subset A$ with $\mu(U) \le \frac{1}{2} \mu(A)$ we have 
$$
    \mu \left(\bigcup_{s\in S} s.U \cap A\right) > (1+\eps) \mu(U).
$$

The relevance of domains of expansion to Question~\ref{intro-q1} stems from the following theorem from~\cite{grabowski2020measurable}.

\begin{theorem}[\cite{grabowski2020measurable}]\label{intro-thm-mot1}
Suppose that $A\subset \RR^3$ is a domain of expansion and let $B\subset \RR^3$ be a measurable set. Then $A$ and $B$ are essentially Lebesgue equidecomposable if and only if $\mu(A)=\mu(B)$ and $A$ and $B$ essentially cover each other.
\end{theorem}

This theorem answers the variant of Question~\ref{intro-q1} for essential Lebesgue equidecompositions satisfactorily when we know that $A$ is a domain of expansion. The notion of domain of expansion was introduced in~\cite{grabowski2020measurable}, where it is shown that it is equivalent to the notion of local spectral gap introduced in~\cite{MR3648974}. Adrian Ioana asked in private communication whether every non-negligible compact subset of $\RR^3$  is a domain of expansion (or equivalently, whether the natural action of $\Iso(\RR^3)$ has the local spectral gap property with respect to every non-negligible compact set). By Theorem~\ref{intro-thm-mot1}, the positive answer to this question would give us a very satisfactory answer to the variant of Question~\ref{intro-q1} for essential Lebesgue equidecompositions: 

\begin{conjecture} Two non-negligible compact subsets of $\RR^3$ are essentially Lebesgue equidecomposable if and only if they have the same measure and they essentially cover each other.
\end{conjecture} 

Alas, our first result shows that this conjecture is false.

\begin{theorem}\label{intro-thm1}
There exist non-negligible compact sets $X,Y\subset \RR^3$ which cover each other, have the same measure, and which are not essentially equidecomposable.
\end{theorem}

In fact we can also answer Adrian Ioana's question in the negative, which is the content of our second result.

\begin{theorem}\label{intro-thm2}
There exists a non-negligible compact set in $\RR^3$ which is not a domain of expansion.
\end{theorem}

\begin{remark}\label{intro-rem-tau}

\part Many subsets of $\RR^3$ are domains of expansion. For example, any bounded subset  of $\RR^3$ with a non-empty interior is a domain of expansion, and~\cite[Theorem  1.13]{grabowski2020measurable} provides an example of a nowhere dense compact set which is a domain of expansion.

\part  The main tool in our proofs of Theorems~\ref{intro-thm1} and~\ref{intro-thm2} is a family of probability measures on the set of all compact subsets contained in the unit cube, constructed through a process which imitates the construction of the Menger sponge (see e.g.~\cite{enwiki:1003864679}).
In fact, we prove that a.e.~compact set is not a domain of expansion, where ``a.e.''~is with respect to one of the probability measures which we construct. Nevertheless we do not know of a non-deterministic construction of a set which is not a domain of expansion. Since our random compact sets can be described as ``random Menger sponges'', it seems reasonable to pose the question whether the standard Menger sponge (or the standard Sierpinski pyramid) is a domain of expansion.

\part 

It is possible to prove a Baire category variant of 
Theorem~\ref{intro-thm1}. Let us say that two sets $X,Y\subset \RR^3$ 
are \emph{$\tau$-essentially equidecomposable} if there exist $X', 
Y'\subset \RR^3$ which are equidecomposable and such that $X\setminus 
X'$, $X'\setminus X$, $Y\setminus Y'$, $Y'\setminus Y$ are of first 
category in, respectively $X$, $X'$, $Y$ and $Y'$. Then it can be 
proven that the compact sets in Theorem~\ref{intro-thm1} are not 
\emph{$\tau$-essentially equidecomposable}. We describe necessary 
modifications in the proof in Remark~\ref{rem-tau-essential}.

\part The analogues of Theorems~\ref{intro-thm1} and~\ref{intro-thm2} hold also in $\RR^n$ for every $n\ge 1$, though they are new only for $n\ge 3$.

\trap
\end{remark}

\paragraph{Outline} In Section~\ref{sec-prelim} we prove some estimates about the binary entropy and define our random set model. The main work is done in Section~\ref{sec-subcongruent}, where we show that a.s.~our random compact set $\KK_\infty$ has the property that for every dyadic cube $K\subset \RR^3$ such that $K\cap \KK_\infty\neq \emptyset$ and every $\ga\in \Iso(\RR^3)$ we have that $\ga(K\cap \KK_\infty) \nsubseteq \KK_\infty$ (Theorem~\ref{thm-subcongruent}). This property is then used in Section~\ref{sec-app} to prove Theorems~\ref{intro-thm1} and~\ref{intro-thm2}.

\paragraph{Acknowledgements} We thank Andr\'as M\'ath\'e  and Oleg Pikhurko for very helpful discussions. In particular they suggested to look at ``generic'' compact sets for examples of sets which are not domains of expansion.

\section{Preliminaries}\label{sec-prelim}

We let $\NN:=\{0,1,\ldots\}$ and $\NN_+:=\{1,2,\ldots\}$. We use the shorthand $(x_i)$ to denote a sequence $(x_i)_{i\in \NN}$. The cardinality of a set $X$ is denoted with $|X|$.

\subsection{Binary entropy estimates}
The word ``logarithm'' is a shorthand for ``base-$2$ logarithm'', and similarly the symbol $\log$ denotes the base-$2$ logarithm.

For a random variable $V$ we let $\range(V)$  be the set of all values which $V$ can take, and we define $\logrange(V) := \log(|\range(V)|)$. If $X$ is a finite set and $V$ is an $X$-valued random variable with law $\nu$ then we let $H(V)$ be the entropy of $V$, i.e.~$H(V) = \sum_{x\in X} -\nu(x)\log(\nu(x))$. For $p\in [0,1]$ we let 
$$
    h(p) := -p\log(p)- (1-p)\log(1-p).
$$ 

Let us recall how the binary entropy $h$ can be used to estimate the binomial coefficients.

\begin{lemma}[{\cite[§X.11, Lemma 7]{MR0465509}}]\label{lem-prenewton}
For $n\in \NN_+$ and $\al \in (0,1)$ such that $\al n\in \NN$ we have 
$$
    \frac{1}{2\sqrt{2}}G \le {n\choose \al n} \le \frac{1}{\sqrt{2\pi}}G,
$$
where 
$$
    G = \frac{2^{h(\al) n}}{ \sqrt{n\al(1-\al)}}.
$$
\end{lemma}

It will be convenient to use the following well-known corollary in the computations.

\begin{lemma}\label{lem-newton}
For $n\in \NN_+$ and $\al\in [0,1]$ such that $\al n\in \NN$, we have
$$
    \frac{2^{h(\al)n}}{n}\le {n\choose \al n} \le 2^{h(\al)n}
$$
\end{lemma}
\begin{proof}
When $\al\in \{0,1\}$, the inequalities become $\frac{1}{n} \le 1 \le 1$, which is true. When $\al\in (0,1)$ and $\al n \in \NN$, we see that necessarily $n\ge 2$. Note also that $\al(1-\al)\le \frac{1}{4}$. Now the inequality
$$
    \frac{2^{h(\al)n}}{n }\le {n\choose \al n}
$$
follows from the previous lemma, since $\frac{1}{2\sqrt{2} \sqrt{n\al(1-\al)}} \ge \frac{1}{\sqrt{2} \sqrt{n}} \ge \frac{1}{n}$.


It remains to show that if $\al\in (0,1)$ and $\al n \in \NN$ (hence $n\ge 2$ and $\frac 1n \le \al \le 1-\frac 1n$), then 
\begin{equation}\label{eq-sd}
 {n\choose \al n} \le 2^{h(\al)n}.
\end{equation}
The value of $\al(1-\al)$ is minimal for $\al=\frac{1}{n}$ and $\al =1-\frac{1}{n}$, and it is equal to $\frac{n-1}{n^2}$. It follows that $\sqrt{2\pi n\al(1-\al)} \ge \sqrt{2\pi \frac{n-1}{n}}$, and since $n\ge 2$, this is greater than $1$. Now the inequality~\eqref{eq-sd} follows from the previous lemma.
\end{proof}

Let us also state the following estimate.

\begin{lemma}\label{lem-ediff}
Let $q\in (0,1)$, $\al \in [0,q]$. We have 
$$
h\left(\frac{q-\alpha}{1-\alpha}\right)(1-\al) \le  h(q) - (1-q)\cdot \al.
$$
\end{lemma}
\begin{proof}

A straightforward calculation shows that 
\begin{equation}\label{todo43}
    h(q)-h\left(\frac{q-\alpha}{1-\alpha}\right)(1-\alpha)=-q\log q - (1-\alpha)\log(1-\alpha) + (q-\alpha)\log(q-\alpha).
\end{equation}

Let us denote the right-hand side of~\eqref{todo43} by $F(\al,q)$. A direct check shows that $\frac{d}{dq} F(\al, q) = \log(\frac{q-\al}{q})$. Since $\log(\frac{q-\al}{q}) \le \frac{q-\al}{q}-1 = -\frac{\al}{q}\le-\al$, we have $\frac{d}{dq} F(\al,q) \le -\al$. Furthermore we have $F(\al,1)=0$, and so by the mean value theorem we have  
$$
    F(\al, q) = F(\al,q) - F(\al,1) \ge (1-q) \cdot \al ,
$$
which establishes the lemma.
\end{proof}

\subsection{Dyadic cubes}

Let $i\in \NN$. We say that $K\subset \RR^3$ is a \emph{dyadic $i$-cube} (or simply an \emph{$i$-cube}) if it is a closed solid cube whose side lengths are $\frac{1}{2^i}$, and whose corners have coordinates of the form $(\frac{a}{2^i}, \frac{b}{2^i}, \frac{c}{2^i})$ , where $a,b,c\in \{0,1,\ldots, 2^i\}$.

A \emph{dyadic cube} is a set which is a dyadic $i$-cube for some $i$.  A \emph{dyadic complex} is a union of finitely many dyadic cubes. A \emph{ dyadic $i$-complex} (or simply an \emph{$i$-complex}) is a dyadic complex which is equal to a union of  dyadic $i$-cubes (thus any dyadic $i$-complex  is also a dyadic $j$-complex for any $j\ge i$).

The unit cube in $\RR^3$ will be denoted by $K_0$.

\begin{remark}
It is not difficult to check that every compact subset of the unit cube is the intersection of a descending sequence of dyadic complexes. 
\end{remark}



\begin{remark}\label{rem-order}
For every $S\in \NN$ we consider the finite set of $S$-cubes to be ordered, with the order induced by the lexicographic order on the coefficients of the midpoints of the cubes. More generally, given a fixed dyadic complex $K$, and $S\in \NN$, the finite set of all $S$-cubes contained in $K$ will be considered  with the order induced from the order just described. 

This will be used in the following fashion: given $i\in \NN$ and a dyadic $i$-complex $K$ which is a union of $m$ distinct $i$-cubes, in order to specify a $j$-subcomplex of $K$, with $j\ge i$, it is enough to specify a sequence of length $m$, whose elements are subsets of the set $\{0,\ldots, 8^{j-i}-1\}$. Indeed, the $k$-the element of such a sequence determines which $j$-cubes in the $k$-th $i$-cube of $K$ to keep. 
\end{remark}

\subsection{Random set model}

\begin{definition}\label{s-growth}
For the remainder of this article we fix $p\in (0,1)$, a sequence $(p_i)$ of rational numbers in $(0,1)$, and a sequence $(s_i)$ of positive natural numbers. For $i\in \NN$  we let $P_i := \prod_{j\le i} p_j$ and  $S_i := \sum_{j\le i} s_j$. We assume that the following conditions hold.
\part  $p_0=1$,
\part $\prod_{i\in \NN} p_i =p$,
\part $\forall i\in \NN$ we have that $p_i8^{s_i}\in \NN$,
\part $\forall i\in \NN$ we have  $(1-p_{i+1})8^{S_i} \ge 7^{S_i}$,
\part $\forall i\in \NN$ we have $3\cdot 8^{S_i} \le s_{i+1}$.
\trap
\end{definition}

\begin{remark}
Clearly the conditions (d) and (e) are satisfied ``as soon as $(s_i)$ grows sufficiently quickly''. In particular for any $(p_i)$ such that $\prod_{i\in \NN} p_i =p$ we can find $(s_i)$ such that the conditions (d) and (e) hold.
\end{remark}

Let $\cal D$ be the set of all compact subsets of $K_0$. We now proceed to define a random variable $(\KK_i)_{i\in \NN}$ with values in  $\prod_\NN \cal D$, with the property that $\forall i\in \NN$ we have that $\KK_{i}$ is an $S_i$-complex and furthermore $\KK_{i+1}\subset \KK_i$. 

We let $\KK_0 := K_0$. Then inductively let us assume that $\mathbb K_i$ has been defined for some $i\in 
\NN$ and that $\mathbb K_i$ is a union of $m_i$ different $S_{i}$-cubes 
$L_0,\ldots, L_{m_i-1}$. Then we subdivide each $L_j$ into $8^{s_{i+1}}$ dyadic $S_{i+1}$-cubes, and for each $j$ we choose uniformly at random exactly $p_{i+1}\cdot 8^{s_{i+1}}$ of those cubes. 




Finally, we let $\mathbb K_\infty:=\bigcap_{i\in \NN} \mathbb K_i$, which is a $\cal D$-valued random variable determined by the random variable $(\mathbb K_i)$. 


\section{No ``internal symmetries'' in $\KK_\infty$}\label{sec-subcongruent}

Let $K$ be a dyadic cube complex, let $S\in \NN$, and let $A$ be a dyadic $S$-cube. We say that $A$ is \emph{subcongruent in $K$} if $\mu(A\cap K)>0$ and there exists $\ga \in \Iso(\RR^3)\setminus \{\Id\}$ such that  $A\cap \ga(A) = \emptyset$ and $\ga(A\cap K)\subset K$. Similarly we say that $K$ \emph{contains a subcongruent $S$-cube} if there exists an $S$-cube $A$ which is subcongruent in $K$.

In this section we prove the following theorem.

\begin{theorem}\label{thm-subcongruent}
$\forall S\in \NN$  $\exists i_0\in \NN$ such that $\forall i>i_0$ we have 
\begin{equation}\label{eq-no-congr}
\Pr(\text{$\mathbb K_i$ contains a subcongruent $S$-cube}) < 2^{-5^{S_i}}
\end{equation}
\end{theorem}

The idea of the proof is a counting argument: we will bound from below the logrange of the variable $\KK_i$, and we will bound from above the logrange of the variable $\KK_i$ under the condition that $\KK_i$ contains a subcongruent $S$-cube.

We start by estimating the logrange of $\KK_i$. 
\begin{proposition}\label{prop-entropy}
\part We have that $\mathbb K_i$ is a union of $m_i :=P_i\cdot 8^{S_i}$  dyadic $S_i$-cubes. In particular we have that $\mu(\KK_i) = P_i$ and $\mu(\KK_\infty)= p$. 
\part For $i\in\NN$ we have
$$
\logrange (\KK_{i+1}) \ge \logrange (\KK_{i})  +  P_{i}\cdot  h(p_{i+1})\cdot 8^{S_{i+1}}  -S_{i+1}^2.
$$
\trap
\end{proposition}

\begin{proof}
\part Let $m_i$ be the number of $S_i$-cubes in $\mathbb K_i$. Clearly $m_0 = 8^{S_0}= p_0\cdot 8^{S_0}$. The definition of $\mathbb K_{i+1}$ from $\mathbb K_i$ requires to subdivide each $S_i$-cube into $8^{s_{i+1}}$ dyadic $S_{i+1}$-cubes and choose $p_{i+1}\cdot 8^{s_{i+1}}$ of them. Thus we get $m_{i+1} = 8^{s_{i+1}}p_{i+1} m_i$. Now the result follows by induction and the fact that $\KK_\infty = \bigcap_{i\in \NN} \KK_i$ .

\part The random variable $\KK_{i+1}$ carries the same information as the pair $(\KK_i, \UU)$, where $\UU$ is a sequence of length $m_i$, whose elements are subsets of the set $\{0,\ldots, 8^{s_{i+1}}-1\}$, of size $p_{i+1} 8^{s_{i+1}}$, chosen independently and uniformly at random (see Remark~\ref{rem-order}). As such we have 
$$
    \logrange(\KK_{i+1}) = \logrange(\KK_i) + m_i\log\left({8^{s_{i+1}}\choose p_{i+1}8^{s_{i+1}}}\right).
$$
After substituting $\al=p_{i+1}$ and $n= 8^{s_{i+1}}$ in Lemma~\ref{lem-newton}, we obtain that 
\begin{align*}
    m_i\log\left({8^{s_i+1}\choose p_{i+1}8^{s_{i+1}}}\right)  &> m_i (h(p_{i+1})8^{s_{i+1}} -3 s_{i+1}) 
\\
&= P_i h(p_{i+1})8^{S_{i+1}} -3 P_i 8^{S_i} s_{i+1}.
\end{align*}
By Definition~\ref{s-growth}(e), the above is greater than
$$
    P_i h(p_{i+1})8^{S_{i+1}} -s_{i+1}^2 \ge P_i h(p_{i+1})8^{S_{i+1}} -S_{i+1}^2 ,
$$
which finishes the proof.

\trap
\end{proof}

\begin{remark}
We now aim towards estimating the logrange of the variable $\KK_i$ under the condition that $\KK_i$ contains a subcongruent $S$-cube.

First we need to make some observations about approximating elements of $\Iso(\RR^3)$. It will be convenient to use the standard representation of elements of $\Iso(\RR^3)$ as $4$-by-$4$ matrices, which we briefly recall now. If $\ga\in \Iso(\RR^3)$ then there is an orthogonal matrix $U$ and a vector $V\in \RR^3$ such that $\forall x\in \RR^3$ we have  $\ga(x) = U(x) + V$. We can consider the $4$-by-$4$ matrix $\begin{pmatrix}U & V \\0 & 1\end{pmatrix}$ and then we have 
$$
\begin{pmatrix}U & V \\0 & 1\end{pmatrix} \begin{pmatrix} x \\ 1\end{pmatrix} = \begin{pmatrix} \ga(x) \\ 1\end{pmatrix}.
$$
\end{remark}

For $n\in \NN$, let $R(n)$ be the set of those rational numbers which can be written as a fraction with denominator $ 2^{n+5}$, and whose absolute value is less than $4$. Let $\Iso(\RR^3, R(n))$ be the set of those isometries which are represented by $4$-by-$4$ matrices with coefficients in $R(n)$. Since $|R(n)| =2^{n+8}-1$, we note that 
$$
    |\Iso(\RR^3,R(n))| \le |R(n)|^{16}< 2^{16n+144}
$$

For $\ga, \de\in \Iso(\RR^3)$ let $d_1(\ga,\de)$ be the maximum of the  absolute values of the entries of the matrix $\ga-\de$. The symbol $d$ will be reserved for the Euclidean distance in $\RR^3$.

\begin{lemma}
Let $n\in \NN$, and let $\ga\in \Iso(\RR^3)$.
\part  If $\ga(K_0) \cap K_0 \neq \emptyset$ then $\exists \bar \ga \in \Iso(\RR^3, R(n))$ such that $d_1(\ga, \bar \ga) < \frac{1}{2^{n+5}}$.
\part If $\de \in \Iso(\RR^3)$ is such that $d_1(\ga, \de) < \frac{1}{2^{n+5}}$ then $\forall x\in K_0$ we have $d(\ga(x), \de(x)) < \frac{1}{4} \frac{1}{2^n}$. 
\trap
\end{lemma}
\begin{proof}

\part Suppose that $V\in \RR^3$ and $U$ is the orthogonal matrix such that  $\forall x\in \RR^3$ we have $\ga(x) = U(x)+V$.  Since the matrix coefficients of $U$ are bounded by $1$ in absolute value, we only need to argue that 
the coefficients of $V$ are bounded by $4$ in absolute value.

To do this let $x\in K_0$ be such that $\ga(x) 
\in K_0$. Since $x\in K_0$, we see that $U(x)$ is contained in $[-2,2]^3$. Since $U(x) +V$ lies in 
$K_0 =[0,1]^3$, we see that all coordinates of $V$ have to be 
bounded by $3<4$ in absolute value, which finishes the proof.

\part  Let $x\in \RR^3$, and let $y\in  \RR^4$ be the vector $(x,1)$. When we act with $\de$ (represented 
as a $4$-by-$4$ matrix) on $y$ then the coefficients of the resulting 
vector differ by less than $4\cdot \frac{1}{2^{n+5}}= \frac{1}{2^{n+3}}$ 
from the respective coefficients of $\ga(y)$. Thus $d(\ga(x), \de(x))< \sqrt{4 \cdot 
(\frac{1}{2^{n+3}})^2} = \frac{1}{2^{n+2}} = 
\frac{1}{4}\frac{1}{2^{n}}$.
\trap
\end{proof}

\newop{inter}{int}
\newop{known}{\textsc known}

Let us deduce how the previous lemma can be applied to dyadic complexes. Let $n\in \NN$, $\ga\in \Iso(\RR^3)$, and let $B$ be an $n$-complex. We let $\inter_n(B)$ be the $(n+2)$-subcomplex of $B$ which consists of those $(n+2)$-cubes which are contained in the topological interior of $B$.

 We define $F_n(\ga,B)$ as the smallest $n$-complex which contains $\ga(\inter_n(B))\cap K_0$ 

\begin{lemma}\label{lem-F}
Let $n\in \NN$, $\ga, \de\in \Iso(\RR^3)$ be such that $d_1(\ga, \de) < \frac{1}{2^{n+5}}$, and let $B$ be an $n$-complex.  
\part  If $K$ is an $n$-cube contained in $F_n(\de,B)$ then $\mu(K\cap \ga(B)) >0$.
\part  If $\ga(B)\subset K_0$ then $\mu(F_n(\de,B) \cap \ga (B)) \ge \frac{1}{8}\mu (B)$.
\trap
\end{lemma}
\begin{proof}
\part  If $K$ is contained in $F_n(\de,B)$, then $\exists x\in K \cap \de(\inter_n(B))$. By the second item of the previous lemma, we have  that $\de (\inter_n(B))$ is contained in the topological interior of $\ga(B)$, and hence for some ball $C$ containing  $x$ we have $C\subset \ga(B)$. Clearly $\mu(C\cap K)>0$, which finishes the proof of (a).

\part By the second item of the previous lemma, $\de(\inter_n(B)) \subset \ga(B)$. By assumption, $\ga(B)\subset K_0$, hence $\de(\inter_n(B))\subset K_0$. It follows that $\de(\inter_n(B))\subset F_n(\de,B)$ and hence
\begin{equation}\label{eq-got39}
    \de(\inter_n(B))\subset F_n(\de,B) \cap \ga(B).
\end{equation}
For any $n$-cube $L$ we have $\mu(\inter_n(L)) = \frac{27}{64}\mu(L)$, so $\mu(\inter_n(B)) \ge  \frac{27}{64}\mu(B) \ge \frac{1}{8}\mu(B)$. This, together with~\eqref{eq-got39} finishes the proof of (b).
\trap
\end{proof}

Let us fix $S\in \NN$ for the rest of this section, and let $C_i$ be the condition  
\begin{center}
``\text{$\mathbb K_i$ contains a subcongruent $S$-cube}'',
\end{center} and let   $\MM_i := \KK_i|C_i$.  We are now ready for the main result needed in the proof of Theorem~\ref{thm-subcongruent}: estimation from above of the logrange of the random variable $\MM_i$.

\begin{proposition}\label{prop-entropy-small}
There exists $i_0\in \NN$ which depends only on $p$ and $S$,  such that for  $i>i_0$  we have
$$
\logrange (\MM_{i+1}) \le \logrange (\MM_{i})  +  8^{S_{i+1}}\cdot P_{i} \cdot h(p_{i+1})- 6^{S_{i+1}}.
$$
\end{proposition} 

\begin{proof}
\newop{UU}{\mathbb U}
\newop{prev}{prev}

We proceed in a similar way as in the proof of Proposition~\ref{prop-entropy}, 
in that we express $\MM_{i+1}$ as a function of a $\MM_i$ and some additional data.

We start by  associating to every element $M$ in the range 
of  $\MM_{i+1}$  a pair  $(A_M,\ga_M)$ which witnesses that $M$ 
contains  a subcongruent $S$-cube, i.e.~$A_M$ is an $S$-cube with 
$\mu(M\cap A_M)>0$ and $\ga_M\in \Iso(\RR^n)$ is such that 
$\ga_M(A_M)\cap A_M=\emptyset$ and $\ga_M(M\cap A_M)\subset M$. 
Furthermore, let $\prev(M)$ be the unique element in the range 
of $\MM_i$ from which $M$ arises, and let $\bar \ga_M\in 
\Iso(\RR^3,R(S_{i+1}))$ be such that $d(\ga_M,\bar \ga_M) \le \frac{1}{2^{S_{i+1}+5}}$.

Now we note that $M$ can be recovered from the tuple
\begin{equation}\label{eq-tup}
    (\prev(M), A_M, \bar \ga_M, M\cap A_M, M\setminus A_M),
\end{equation}
since clearly $M = (M\cap A_M) \cup (M \setminus A_M))$. In order to estimate the logrange of $\MM_{i+1}$, we proceed to bound from above the number of such tuples, as follows. In the following claim, we fix the first 4 elements in~\eqref{eq-tup}, and we estimate the number of elements which we can put in the fifth place. Afterwards, we will estimate the number of possibilities for the first 4 elements in~\eqref{eq-tup}. These two bounds will lead to the desired result.

\begin{claim} Let $A\in \range (\MM_i)$, let $B$ be an $S$-cube, let $\ga\in \Iso(\RR^3,R(S_{i+1}))$, and let $D$ be an $S_{i+1}$-complex contained in $B$. Then the logarithm of the number 
of all dyadic $S_{i+1}$-complexes $E$ such that for some $M\in \range(\MM_{i+1})$ we have that 
$$
(A,B,\ga,D, E) = (\prev(M),A_M,\bar \ga_M, M\cap A_M, M\setminus A_M))
$$
is bounded from above by
\begin{equation}\label{eq-ssa}
    8^{S_{i+1}}\cdot \mu(\prev(M)\setminus A_M) \cdot h(p_{i+1})- \frac{p_{i+1}}{8^{S+1}} \cdot 8^{S_{i+1}}(1-p_{i+1})
\end{equation}
\end{claim}
\begin{proof}[Proof of Claim]
Given $(\prev(M), A_M,\bar \ga_M, M\cap A_M)$, in order to describe $M\setminus A_M$ we proceed in a similar fashion as in Proposition~\ref{prop-entropy}. 

 The crucial difference is that when we choose which $S_{i+1}$-cubes should be in $M\setminus A_M$, we can start by including all $S_{i+1}$-cubes which are contained in $F_{S_{i+1}}(\bar \ga_M,M\cap A_m)$, and the latter set is determined by $(B,\ga, D)$. This observation is what leads to the large negative term in~\eqref{eq-ssa}.

For any $S_i$-cube $K$  we let 
$$
    \al(K) := \mu(F_{S_{i+1}}(\bar \ga_M, M\cap A_M) \cap K) \cdot 8^{S_i}
$$
In other words, $\al(K)$ is the ``density'' of $F_{S_{i+1}}(\bar \ga_M, M\cap A_M)$ in $K$. Informally, the number $\al(K)8^{s_{i+1}}$ represents those $S_{i+1}$-cubes of $K$ which we must choose in every $E$, i.e. we can ``deduce the presence of those dyadic cubes from the information contained in $(A,B,\ga,D)$''.   

Since $\ga(M\cap A_M) \subset M\setminus A_M$, we have

\begin{align*}
\sum_{\substack{K\subset M\setminus A_M\\ \text{$K$ is an $S_i$-cube}}} \al(K) &= \mu(F_{S_{i+1}}(\bar \ga_M,  M\cap A_M) \cap (M\setminus A_M))  \cdot 8^{S_i} 
\\
&\ge \mu(F_{S_{i+1}}(\bar \ga_M,  M\cap A_M) \cap \ga_M(M\cap A_M))  \cdot 8^{S_i},
\end{align*}
which by the second item of Lemma~\ref{lem-F} is greater or equal to 
$$
\frac{1}{8}\mu(M\cap A_M)  \cdot 8^{S_i}
\ge \frac{p_{i+1}}{8^{S+1}}\cdot 8^{S_i}.
$$

Note that there are  $m_i' := \mu(M\setminus A_M)\cdot 8^{S_i}$ dyadic $S_i$-cubes contained in $M\setminus A_M$. In order to describe $M\setminus A_M$  we need a sequence of length $m_i'$,  such that the element corresponding to a given $S_i$-cube $K$ contained in $M\setminus A_M$ is a subset of cardinality $(p_{i+1}-\al(K))8^{s_{i+1}}$ of the set \{0,1,\ldots, $(1-\al(K))8^{s_{i+1}}-1\}$ (see Remark~\ref{rem-order}). As such there are at most 
\begin{equation}\label{eq-gg}
\prod_{j < m_i'} {(1-\al_j)8^{s_{i+1}} \choose (p_{i+1}-\al_j)8^{s_{i+1}}}.
\end{equation}
possibilities for the choice of $E$, where $(\al_j)$ is some sequence of length $m_i'$ with $\sum_{j<m_i'} \al_j \ge \frac{p_{i+1}}{8^{S+1}} 8^{S_i}$.

By Lemma~\ref{lem-newton}, the logarithm of~\eqref{eq-gg} is bounded from above by
\begin{equation}\label{eq-pl}
    \sum_{j<m_i'} h\left(\frac{p_{i+1}-\al_j}{1-\al_j}\right)(1-\al_j) 8^{s_{i+1}}
\end{equation}

Putting $q=p_{i+1}$, $\alpha=\alpha_j$  in Lemma~\ref{lem-ediff}, and summing over $j<m_i'$, we see that~\eqref{eq-pl} is bounded from above by
\begin{align*}
    8^{s_{i+1}}&\left(m_i' h(p_{i+1})- \sum_{j<m_i'} \alpha_j(1-p_{i+1}) \right) 
\\
 &\le     8^{s_{i+1}}\left(m_i' h(p_{i+1})- \frac{p_{i+1}}{8^{S+1}} \cdot 8^{S_i}(1-p_{i+1}) \right) 
\\
&= 8^{s_{i+1}}\cdot m_i' h(p_{i+1})- \frac{p_{i+1}}{8^{S+1}} \cdot 8^{S_{i+1}}(1-p_{i+1}).
\end{align*}
Since $m_i' = \mu(M\setminus A_M)\cdot 8^{S_i} < \mu(\prev(M)\setminus A_M)\cdot 8^{S_i}$, this proves the claim.
\end{proof}

\begin{claim} The logarithm of the  number of tuples $(A,B,\ga, D)$ such that for some $M\in \range(\MM_{i+1})$ we have 
$$
    (A,B,\ga, D) = (\prev(M),A_M,\bar \ga_M, M\cap A_M)
$$
is bounded from above by
$$
    \logrange(\MM_i) + 8^{S_{i+1}}\mu(\prev(M)\cap A_M) \cdot  h(p_{i+1})  +  16 S_{i+1} + S + 144
$$
\end{claim}
\begin{proof}
Since $A\in \range(\MM_i)$, we have at most $|\range(\MM_i)|$ possibilities for it. Since $A_M$ is an $S$-cube, we have at most $8^S$ possibilities for it, and since $\bar\ga_M\in \Iso(\RR^3,R(S_{i+1}))$, we have $2^{16S_{i+1}+144}$ possibilities for it.

Finally to describe $D$, we proceed exactly as in Proposition~\ref{prop-entropy}, i.e.~$A_M\cap \prev(M)$ consists of $\mu(A_M\cap\prev(M))\cdot 8^{S_i}$ dyadic $S_i$-cubes, and so we only need a sequence of length $\mu(A_M\cap\prev(M))8^{S_i}$, whose each element is a subset of cardinality $ p_{i+1}8^{s_{i+1}}$ of the set $8^{s_{i+1}}$. Thus, for a fixed $A,B$ there are at most 
$$
 {8^{s_{i+1}}\choose p_{i+1}8^{s_{i+1}}}^{\mu(A_M\cap\prev(M))8^{S_i}}
$$
possibilities for $D$. Multiplying all the relevant factors together and taking the logarithm of the result we obtain 
$$\logrange(\MM_i)+S+16S_{i+1}+144+\mu(\prev(M)\cap A_M)8^{S_i}\log {8^{s_{i+1}}\choose p_{i+1}8^{s_{i+1}}},$$
which by Lemma~\ref{lem-newton} does not exceed
$$
    \logrange(\MM_i) + 8^{S_{i+1}}\mu(\prev(M)\cap A_M) \cdot  h(p_{i+1})  +  16 S_{i+1} + S + 144,
$$
which finishes the proof of the claim.
\end{proof}

Note that for any $M\in\range(\MM_{i+1})$ we have $\mu(\prev(M)) = P_i$. As such, the previous two claims show together that the logarithm of the  number of possibilities for the tuples as in~\eqref{eq-tup} is bounded from above by 
$$
    \logrange(\MM_i) + 8^{S_{i+1}}\cdot P_i \cdot  h(p_{i+1}) + 16 S_{i+1} + S + 144 - \frac{p}{2^{S+1}} \cdot 8^{S_{i+1}}(1-p_{i+1}).
$$

By Definition~\ref{s-growth}(d), we have 
$$
    16 S_{i+1} + S + 144 - \frac{p}{2^{S+1}} \cdot 8^{S_{i+1}}(1-p_{i+1}) \le  16 S_{i+1} + S + 144 - \frac{p}{2^{S+1}} \cdot 7^{S_{i+1}}.
$$ 
Since  $(S_i)$ is an increasing sequence, we see that there exists $i_0$ such that for $i>i_0$ we have 
$$
 16 S_{i+1} + S + 144 - \frac{p}{2^{S+1}} \cdot 7^{S_{i+1}} \le - 6^{S_{i+1}}
$$
Since $\logrange(\MM_{i+1})$ is bounded from above by the number of tuples as in~\eqref{eq-tup}, we see that all in all for $i>i_0$ we have 
$$
\logrange(\MM_{i+1})  \le   \logrange(\MM_i) + 8^{S_{i+1}}\cdot P_i \cdot  h(p_{i+1}) - 6^{S_{i+1}},
$$
which finishes the proof.
\end{proof}

We have now everything in place for the proof of Theorem~\ref{thm-subcongruent}.

\begin{proof}[Proof of Theorem~\ref{thm-subcongruent}]
Recall that $S\in \NN$ is fixed, furthermore $C_{i+1}$ is the condition  
\begin{center}
``\text{$\mathbb K_{i+1}$ contains a subcongruent $S$-cube}'', 
\end{center}
and we defined $\MM_{i+1} = \KK_{i+1}|C_i$.  

We have
$$
\Pr(C_{i+1})  = \frac{|\range(\MM_{i+1})|}{|\range(\KK_{i+1})|} = 2^{\logrange(\MM_{i+1})-\logrange(\KK_{i+1})}.
$$
 By Proposition~\ref{prop-entropy-small} there exists $i_0$ such that for $i>i_0$ we have
\begin{align*}
\logrange (\MM_{i+1}) &\le \logrange (\MM_{i})  +  8^{S_{i+1}}\cdot P_{i} \cdot h(p_{i+1})- 6^{S_{i+1}}
\\
&\le \logrange (\KK_{i})  +  8^{S_{i+1}}\cdot P_{i} \cdot h(p_{i+1})- 6^{S_{i+1}}.
\end{align*}
By Proposition~\ref{prop-entropy}, for all $i\in \NN$ we have 
$$
    \logrange (\KK_{i+1}) \ge \logrange (\KK_{i})  +  8^{S_{i+1}} \cdot P_{i}\cdot  h(p_{i+1})  -S_{i+1}^2.
$$
As such,  we have
$$
\logrange(\MM_{i+1})-\logrange(\KK_{i+1}) \le -6^{S_{i+1}} + S_{i+1}^2,
$$
which is less than $-5^{S_{i+1}}$ for large enough $i$. This finishes the proof.
\end{proof}

\section{Applications to domains of expansion and equidecomposability}\label{sec-app}

In this section we prove Theorems~\ref{intro-thm1} and~\ref{intro-thm2}. Let $\cal R$ be the subset of $\range( (\KK_i)_{i\in \NN} )$ consisting of those sequences $(M_i)$ such that for every $S\in \NN$ there exists $i_0\in \NN$ such that for $i\ge i_0$ there are no subcongruent $S$-cubes in $M_i$. By Theorem~\ref{thm-subcongruent}, $\cal R$ has full measure. 

For subsets $X$ and $Y$ of $\RR^3$ we shall write $X=^*Y$ if $\mu(X \triangle Y) = 0$, where $\triangle$ denotes the symmetric difference of sets. We write $X\ne^*Y$ if $\neg(X=^*Y)$. 

\begin{lemma}\label{lem-key}
For all $(M_i)\in\cal R$, all dyadic complexes $A$ such that $A \cap M_\infty \ne^*\emptyset$, and all finite $T\subset\Iso(\RR^3)\setminus \{\Id\}$ there exists a dyadic cube $K\subset A$ and $j\in\NN$ such that $K\cap M_\infty\ne^*\emptyset$ and $\forall \de\in T$ we have  $M_j \cap \de(K) = \emptyset$. In particular, for all $k\ge j$ we have $M_k \cap \de(K) = \emptyset$. 
\end{lemma}

\begin{proof}
The proof is by induction on $n=|T|$.

Let us fix $(M_i)\in\cal R$ and a dyadic complex $A$ with $A \cap M_\infty \ne^*\emptyset$. Let $n=0$, 
i.e.~we have $T=\emptyset$. Since $A \cap M_\infty \ne^*\emptyset$, there exists a dyadic cube $K\subset A$ with $K \cap M_\infty \ne^*\emptyset$. The second condition is vacuously true for arbitrary $j$.

Let us suppose now that the statement holds for some $n\ge 0$. Let $T\subset\Iso(\RR^3)\setminus\{\Id\}$ be a subset of cardinality $n+1$. Let us write $T=U\cup \{\ga\}$ for some $U$ with $|U|=n$. 
By the inductive hypothesis, there exists a dyadic cube $K\subset A$ and $j\in \NN$ such that $K\cap M_\infty\ne^*\emptyset$ and $M_j \cap \de(K) =\emptyset$ for all $\de\in U$.

Let us choose a dyadic cube $K'\subset K$ such that $K'\cap \ga(K') =\emptyset$ and $K'\cap M_\infty\neq^* \emptyset$. Since $(M_i)\in \cal R$, for sufficiently large $k$ we have that $K'$ is not subcongruent in $M_k$. Fix such a number $k>j$. Then $\ga(K'\cap M_k)\setminus M_k \ne^* \emptyset$. Therefore we can choose a dyadic cube $K''\subset K'$ such that $K'' \cap M_\infty\neq^* \emptyset$ and $\ga(K'')\cap M_k =\emptyset$. Then $K''$ and $k$ witness that the inductive statement holds for $T$.

The statement ``In particular...'' follows from the fact that $(M_i)$ is a descending sequence of sets. This finishes the proof.
\end{proof}

The following theorem has Theorem~\ref{intro-thm2} as a corollary.

\begin{theorem}
For every $(M_i)\in\cal R$ the set $M_\infty$ is a compact set of positive measure which is not a domain of expansion.
\end{theorem}

\begin{proof}
The statement about positive measure is Proposition~\ref{prop-entropy}(a).

For the sake of contradiction, let us assume that $(M_i)\in\cal R$ is such that $M_\infty$ is a domain of expansion. Let us fix a finite set $S \subset \Iso(\RR^3)$ and $\eps>0$. To obtain a contradiction we have to find a measurable set $U\subset M_\infty$ with $0<\mu(U)\le \frac 12 \mu(M_\infty)$ and 
\begin{equation}\label{eq-not-domofexp}
\mu(M_\infty \cap \bigcup_{\de\in S} \de(U)) \le \mu(U)(1+\eps). 
\end{equation}

The application of Lemma \ref{lem-key} for $(M_i)$, $A=K_0$, and $T=S\setminus \{\Id\}$ gives us a dyadic cube $K$ and an integer $i$ such that $K \cap M_\infty \ne^*\emptyset$ and $M_i \cap \de(K) =\emptyset$ for all $\de\in S \setminus \{\Id\}$. We may assume that $\mu(K \cap M_\infty)\le\frac 12 \mu(M_\infty)$ (if this is not the case, then we can replace $K$ by a suitably chosen subset). Let $U=K\cap M_\infty$. Since $M_\infty \subset M_i$ and $U \subset K$, it follows that $M_\infty \cap \de(U) = \emptyset$ for all $\de\in S \setminus \{\Id\}$. As a consequence, the set $M_\infty \cap \bigcup_{\de\in S} \de(U)$ is either empty (if $\Id\notin S$) or equals $U$ (otherwise). In both cases the inequality~\eqref{eq-not-domofexp} is clearly satisfied, which finishes the proof.
\end{proof}

We now proceed to the proof of Theorem~\ref{intro-thm1}. The idea is to consider sets of the form $X=C\sqcup C\sqcup D$ and $Y=C \sqcup D \sqcup D$ where $C$ and $D$ are independent copies of the random set $\KK_\infty$, and prove that in any equidecomposition between $X$ and $Y$ there exists an element $x$ of the two copies of $C$ in $X$ which has to be mapped into the same element in $Y$, which is impossible in an equidecomposition. 

We start with the following standard lemma.

\begin{lemma}\label{lem-ess-reduction}
Suppose that $X,Y\subset \RR^3$ are measurable sets which are essentially equidecomposable. Then there exist measurable sets  $X'\subset X$, $Y'\subset Y$ which are equidecomposable and such that $\mu(X\setminus X') = \mu(Y \setminus Y')= 0$.
\end{lemma}
\begin{proof}

Since $X$ and $Y$ are assumed to be essentially equidecomposable, there 
exist measurable sets $X_0$ and $Y_0$ which are equidecomposable and 
such that $\mu(X\triangle X_0) = \mu(Y\triangle Y_0 ) = 0$. Let us 
define $X_1  := X \cap X_0$, $Y_1 := Y \cap Y_0$, and let $R_0 := 
X_0\setminus X_1$, $S_0 := Y_0\setminus Y_1$. 

Finally suppose that $X_0 = \bigsqcup_{i=1}^n A_i$,  $Y_0= \bigsqcup_{i=1}^n B_i$, where $B_i = \ga_i(A_i)$ for some $\ga_i \in \Iso(\RR^3)$. Let $\Ga$ be the group generated by $\ga_1,\ldots, \ga_n$. 

Let $T = \bigcup_{\ga \in \Ga} \ga(R_0\cup S_0)$. Since $\mu(R_0) = \mu(S_0) = 0$, we also have $\mu(T) =0$. Let $X' := X\setminus T$, $Y' := Y\setminus T$. Since $T$ is $\Ga$-invariant, we have 
$$
    \ga_i(A_i\setminus T ) = B_i\setminus T
$$
for all $i$, and hence  $X'=\bigsqcup_{i=1}^n (A_i\setminus T)$ and $Y'= \bigsqcup_{i=1}^n (B_i\setminus T)$ is an equidecomposition between $X'$ and $Y'$. This finishes the proof.

\end{proof}

We are now ready to prove Theorem~\ref{intro-thm1}.

\begin{proof}[Proof of Theorem~\ref{intro-thm1}]
Let us fix $(M_i)\in\cal R$. Then some $M_j$ contains two disjoint $S_j$-cubes, call them $A$ and $B$. Let $\al\in\Iso(\RR^3)$ be the translation by the vector $(1,0,0)$. Define $C=A\cap M_\infty$, $D=B\cap M_\infty$, $X=C \cup D \cup \al(C)$, and $Y=C \cup D \cup \al(D)$. 

Clearly $X$ and $Y$ are both compact sets of the same positive measure, and we have that $X$ and $Y$ cover each other.  Thus, to finish the proof we need to show that $X$ and $Y$  are not essentially equidecomposable.

For the sake of contradiction, suppose that $X$ and $Y$ are essentially equidecomposable.  By Lemma~\ref{lem-ess-reduction}, we can find measurable sets  $X'\subset X$ and $Y'\subset Y$ which are equidecomposable and such that  $\mu(X\setminus X') = \mu(Y \setminus Y')= 0$. Let us write $X' = \bigsqcup_{i=1}^n X_i$,   $Y' = \bigsqcup_{i=1}^n Y_i$, where for all $i$ we have $Y_i = \ga_iX_i$ for some $\ga_i\in \Iso(\RR^3)$. Without loss of generality we can assume that $X_i\neq \emptyset$ for all $i$. 

Let 
$$
    T=\bigcup_{i=1}^n \{\ga_i, \al^{-1}\ga_i, \al^{-1}\ga_i\al,\ga_i\al\}.
$$

We apply Lemma~\ref{lem-key} for $(M_i)$, $T\setminus\{\Id\}$ and $A$ 
to obtain a cube $K\subset A$ and $k\in \NN$ such that $\mu(K\cap 
M_\infty)>0$ and $M_\infty \cap \de(K) = \emptyset$ for all $\de\in 
T\setminus\{\Id\}$. 

By subdividing the sets $X_i$ if necessary, we may assume that for all $i$ we have either $K\cap X_i =\emptyset$ or $\al(K)\cap X_i=\emptyset$. This does not lead to any circularity in the definition of $K$ since $K$ depends only on the set $T$, and the set $T$ does not change after subdividing the sets $X_i$.

With this in mind, we note that the sets 
$$
    K\cap M_\infty \cap \bigcup_{\substack{1\le i,j\le n\\i\neq j}} (X_i \cap \al^{-1}(X_j))
$$ 
and $K\cap M_\infty$ have the same positive measure.   It follows that for some $\ell$ and $m$  with $\ell\neq m$ we have  $X_\ell \cap \al^{-1}(X_m) \cap K\cap M_\infty \neq \emptyset$. Thus there exists $z\in K\cap M_\infty$ such that $z\in X_\ell$ and $\al(z) \in X_m$. 

To obtain the desired contradiction, we shall now prove that $\ga_\ell(z)=z=\ga_m(\al(z))$.

\begin{claim} We have $\ga_\ell = \Id$.
\end{claim}
\begin{proof}[Proof of Claim]
By way of contradiction, let us assume that  $\ga_\ell \neq \Id$. Then by the choice of $K$ we have 
 $\ga_\ell(K)\cap M_\infty =\emptyset$.  Since $z\in K$ and $C\cup D \subset M_\infty$, it follows that $\ga_\ell(z) \notin  C\cup D$.

Since $z\in C$, we have $\al(z)\notin  Y$, and hence we deduce that $\ga_\ell \neq \al$, 
and so $\al^{-1}\ga_\ell \neq \Id$. By the choice of $K$, we have that 
$\al^{-1}\ga_\ell(K) \cap M_\infty = \emptyset$. Since $z\in K$ and 
$D\subset M_\infty$, it follows that $\al^{-1}\ga_\ell(z)\notin D$, hence $\ga_\ell(z)\notin \al(D)$.

All in all, we have that  $\ga_\ell(z)\notin C \cup D \cup \al(D) = Y$, which is impossible as $\ga_\ell(z) \in \ga_\ell (X_\ell) \subset Y$. This contradiction shows that $\ga_\ell=\Id$, which finishes the proof of the claim.
\end{proof}

\begin{claim} We have $\ga_m\al =\Id$.
\end{claim}

\begin{proof}[Proof of Claim]
By way of contradiction, let us assume $\ga_m\al\neq \Id$. Then by the choice of $K$ we have  $\ga_m\al(K)\cap M_\infty = \emptyset$.  Since $z\in K$ and $C\cup D \subset M_\infty$, it follows that $\ga_m\al(z)\notin C\cup D$. 

Since $\ga_m\al(z) \in Y$, we deduce that $\ga_m\neq \Id$, and so $\al^{-1}\ga_m\al\neq \Id$. By the choice of $K$, we have that $\al^{-1}\ga_m\al(K) \cap M_\infty = \emptyset$. Since $z\in K$ and $D\subset M_\infty$, it follows that $\al^{-1}\ga_m\al(z)\notin D$, hence $\ga_m\al(z)\notin \al(D)$.

All in all, we have that $\ga_m\al(z) \notin C \cup D \cup \al(D) = Y$, which is impossible as $\ga_m\al(z)\in  \ga_m(X_m) \subset Y$. This contradiction shows that $\ga_m\al=\Id$, which finishes the proof of the claim.
\end{proof}

Thus we have  $\ga_\ell(z)=z=\ga_m(\al(z))$ and so the sets $Y_\ell=\ga_\ell(X_\ell)$ and $Y_m=\ga_m(X_m)$ are not disjoint. This contradicts the definition of $Y_\ell$ and $Y_m$, and shows that $X$ and $Y$ are not essentially equidecomposable. This finishes the proof.
\end{proof}

\begin{remark}\label{rem-tau-essential}
Let us explain the necessary changes in the proof of Theorem~\ref{intro-thm1} to show that the sets $X$ and $Y$ defined in that proof are not $\tau$-essentially equidecomposable, as defined in Remark~\ref{intro-rem-tau}. First, by imitating Lemma~\ref{lem-ess-reduction}, we can find subsets $X'$ and $Y'$ which are equidecomposable and such that $X\setminus X'$ and $Y\setminus Y'$ are of first category in, respectively, $X$ and $Y$. 

The only other change needed in the proof is arguing why the set
\begin{equation}\label{eq-33}
    K\cap M_\infty \cap \bigcup_{\substack{1\le i,j\le n\\i\neq j}} (X_i \cap \al^{-1}(X_j))
\end{equation}
is non-empty. We have that the complement of the set
$$
X\cap \bigcup_{\substack{1\le i,j\le n\\i\neq j}} (X_i \cap \al^{-1}(X_j))
$$
in the compact set $X = C \cup D \cup \al(C)$ is of first category. On the other hand, the topological interior of $K\cap M_\infty$ in $X$ is non-empty. Since $X$ is in particular a Baire space, we deduce that~\eqref{eq-33} is indeed non-empty.
\end{remark}

\bibliography{bibliografia.bib}

\end{document}